\documentclass[12pt]{article}

\usepackage{amssymb}
\usepackage{amsmath}
\usepackage{amsthm}
\usepackage{txfonts}
\usepackage{enumerate}
\usepackage{graphicx}
\usepackage{xcolor}
\usepackage[top=30truemm, bottom=30truemm, left=20truemm, right=20truemm]{geometry}
\usepackage{comment}

\newtheorem{thm}{Theorem}[section]
\newtheorem{lem}{Lemma}[section]

\title{{\LARGE Pattern block method for generating random numbers : 
Reformulation and generalization of the Ziggurat method using conditional random variables}}
\author{Kensuke Ishitani
\thanks{Department of Mathematical Sciences, Tokyo Metropolitan University, Hachioji, Tokyo 192-0397, Japan} 
\thanks{E-mail: k-ishitani@tmu.ac.jp
\qquad https://orcid.org/0000-0001-8400-0654}
\and 
Ryusei Gomi
\thanks{Department of Mathematical Sciences, Tokyo Metropolitan University, Hachioji, Tokyo 192-0397, Japan}
\and
Atsuya Kagawa
\thanks{Department of Mathematical Sciences, Tokyo Metropolitan University, Hachioji, Tokyo 192-0397, Japan}
}
\date{}
\begin{document}
\maketitle

\begin{abstract}
The Ziggurat method is an efficient rejection sampling technique for generating one-dimensional normally distributed random numbers. 
This study proposes the pattern block method, a generalization of the Ziggurat method.
The pattern block method enables the generation of random numbers 
from multimodal density functions and multidimensional distributions. 
The effectiveness of the pattern block method is demonstrated through several examples. 

\end{abstract}

\section{Introduction}

Random number generation is a fundamental component of Monte Carlo methods. 
Recently, \cite{bib_Jalalvand, bib_Marsaglia_4, bib_Marsaglia_3, bib_McFarland, bib_Su, bib_Zhang} 
developed the Ziggurat method for generating $1$-dimensional normally distributed random numbers 
and demonstrated its effectiveness using numerical examples.
However, the random variable underlying the Ziggurat algorithm has not yet been formally defined. 
Therefore, this study reformulates and generalizes the Ziggurat method using conditional random variables. 
Conditional random variables play an important role in  
\cite{bib_Durret_1977, bib_Ishitani_1, bib_Ishitani_2, bib_Ishitani_3}.

The remainder of this paper is organized as follows.
Section~\ref{Sec_MainResults} presents the pattern block method proposed in this study.
Sections~\ref{Sec_Examples_1dim} and \ref{Sec_Examples_2dim} present applications of the pattern block method.
Finally, Section~\ref{Sec_conclusion} concludes.

\section{The pattern block method}\label{Sec_MainResults}

Let $(E,\mathcal{E}, \mu)$ be a $\sigma$-finite measure space. 
Let $f : E \rightarrow \mathbb{R}$ be a $\mathcal{E}/\mathcal{B}(\mathbb{R})$-measurable function that satisfies
\begin{align}\label{eq_f_densityfunction}
f(x)\geq 0\ (x\in E),\qquad K:=\int_{E}f(x)\, \mu(dx)\in (0, \infty),
\end{align}
where $\mathcal{B}(\mathbb{R})$ denotes the Borel $\sigma$-algebra of $\mathbb{R}$. 
Then $\displaystyle \frac{1}{K}f(x)$ is a probability density function. 
Let
\begin{align*}
\Gamma[f]:=
\left\{(x,y)\in E \times \mathbb{R} \ \big\vert \ 0\leq y\leq f(x)\right\}
\in \mathcal{E} \otimes \mathcal{B}(\mathbb{R}), 
\end{align*}
where $\mathcal{E} \otimes \mathcal{B}(\mathbb{R})$ denotes 
the product $\sigma$- algebra of $\mathcal{E}$ and $\mathcal{B}(\mathbb{R})$. 
Let $\nu$ be the measure on $(E\times \mathbb{R}, \mathcal{E} \otimes \mathcal{B}(\mathbb{R}))$  
\begin{align}\label{def_measure_nu}
\nu(C):= \iint_{E\times \mathbb{R}} 1_C(x,y)\, \mu(dx)\, dy
\ \, (C\in \mathcal{E} \otimes \mathcal{B}(\mathbb{R})).
\end{align}
Let $\widetilde{\pi}: E \times \mathbb{R} \to E$ be a projection defined by
\begin{align*}
\widetilde{\pi}(x, y) :=  x\quad ((x,y)\in E\times \mathbb{R}).
\end{align*}

Let $N\in \mathbb{N}$. 
We assume that $B_1, B_2, \cdots, B_N \in \mathcal{E} \otimes \mathcal{B}(\mathbb{R})$ 
and $\displaystyle B:= \bigcup_{i=1}^N B_i$ 
satisfy the following three conditions:
\begin{align}
&\nu(B_i) > 0 \quad (i=1,2,\cdots, N),
\label{Block_condition_1}\\
&\nu(B_i\cap B_j) = 0 \quad (i\neq j), 
\label{Block_condition_2}\\
&\nu\left(\Gamma[f] \setminus B\right) = 0, \quad \nu(B)<\infty .
\label{Block_condition_3}
\end{align}
We refer to $B_1, B_2, \cdots, B_N$ as the pattern blocks of $f(x)$.
Using \eqref{eq_f_densityfunction}, \eqref{Block_condition_2}, and \eqref{Block_condition_3}, 
we obtain
\begin{align}
K=\nu(\Gamma[f]) \leq \nu(B)= \sum_{i=1}^N \nu(B_i)<\infty.
\label{ineq_nuB_is_greater_than_1}
\end{align}

\begin{figure}[h]
\begin{center}
\includegraphics[scale=0.6]{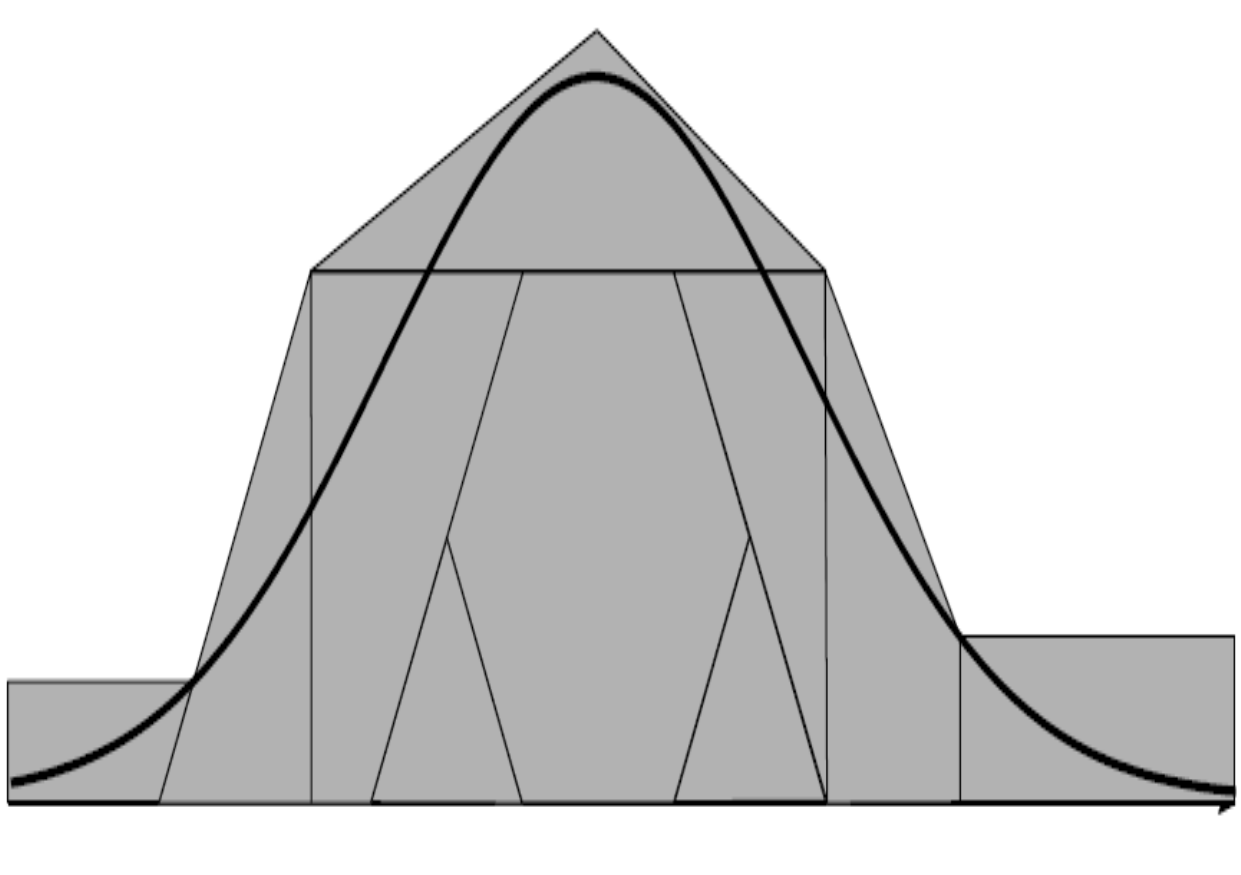}
\caption{$f(x)$ and its pattern blocks}
\label{fig:patternblock}
\end{center}
\end{figure}

Let $\Psi$ be the block selection function defined as
\begin{align*}
\Psi (u):= \min \left\{i \in \{1, 2, \cdots, N\}\ \middle|\ u < \sum_{j=1}^{i}\frac{\nu(B_j)}{\nu(B)}\right\} 
\qquad (0\leq u<1).
\end{align*}
For $i=1, 2, \cdots, N$, let
\begin{align*}
B_i[f]:= B_i \cap \Gamma[f]=\left\{(x, y)\in B_i \ |\ 0\leq y\leq f(x)\right\}.
\end{align*}
Let $\Lambda$ be the acceptance set defined as
\begin{align}
&\Lambda :=
\big\{(u, v_1, v_2, \cdots, v_N, w_1, w_2, \cdots, w_N)  
\in [0,1) \times E^N \times \mathbb{R}^N\ \big|\ (v_{\Psi (u)}, w_{\Psi (u)})\in B_{\Psi (u)}[f] \big\}.
\label{def_set_Lambda} 
\end{align}
Let 
$\pi_{\Psi}: [0,1) \times E^N \times \mathbb{R}^N \to E$ denote the projection:
\begin{align}
&\pi_{\Psi}(u, v_1, v_2, \cdots, v_N, w_1, w_2, \cdots, w_N) 
 := \widetilde{\pi}(v_{\Psi (u)}, w_{\Psi (u)})=v_{\Psi (u)}
\label{def_pi_I_random_projection}\\
&((u, v_1, v_2, \cdots, v_N, w_1, w_2, \cdots, w_N)\in [0,1) \times E^N \times \mathbb{R}^N).
\nonumber
\end{align}

Let $(\Omega, \mathcal{F}, P)$ be a probability space. 
Let $U, W_1, W_2, \cdots, W_N$ be $1$-dimensional random variables on $(\Omega, \mathcal{F}, P)$,
and let $V_1, V_2, \cdots, V_N$ be $(E, \mathcal{E})$-valued random variables on $(\Omega, \mathcal{F}, P)$. 
For $i=1, 2, \cdots, N$, we assume that 
$(V_i, W_i)$ and $U$ are independent and satisfy
\begin{align}
&P((V_i, W_i)\in C)=\dfrac{\nu(C\cap B_i)}{\nu(B_i)}
\quad (C\in \mathcal{E} \otimes \mathcal{B}(\mathbb{R})), 
\label{eq_Dist_Vi_Wi}\\
&P\left(U\in D\right)=\int_D 1_{[0,1)}(u)\, du
\qquad (D\in \mathcal{B}(\mathbb{R})).
\label{eq_Dist_U}
\end{align}
Let
$Z: (\Omega, \mathcal{F}, P) \rightarrow 
(\mathbb{R} \times E^N \times \mathbb{R}^N,
\ \mathcal{B}(\mathbb{R}) \otimes \mathcal{E}^{\otimes N} \otimes \mathcal{B}(\mathbb{R}^N))$
be a random variable given by 
\begin{align}\label{Def_RV_Z_on_OFP}
Z:=(U, V_1, V_2, \cdots, V_N, W_1, W_2, \cdots, W_N).
\end{align}
Then we obtain the following lemma.
\begin{lem}\label{lem:main}
For $A\in \mathcal{E}$, we have
\begin{align}\label{lem_main_eq}
P(V_{\Psi (U)}\in A, Z \in  \Lambda)
=\frac{1}{\nu(B)}\int_{A} f(x)\ \mu(dx).
\end{align}
\end{lem}
\begin{proof}
Since $(V_i, W_i)$ and $U$ are independent for every $i=1, 2, \cdots, N$, we obtain
\begin{align}
&P(V_{\Psi (U)}\in A, Z \in  \Lambda) 
= \sum_{i=1}^NP(\Psi (U)=i)\cdot P((V_i,W_i)\in B_i[f]\cap (A\times \mathbb{R})).
\label{lem_proof_step1} 
\end{align}
From \eqref{Block_condition_2}, \eqref{eq_Dist_Vi_Wi}, \eqref{eq_Dist_U}, and \eqref{lem_proof_step1}, we obtain
\begin{align}
P(V_{\Psi (U)}\in A, Z \in  \Lambda) 
&= \sum_{i=1}^N\frac{\nu(B_i)}{\nu(B)}\cdot E[1_{B_i[f] \cap (A\times \mathbb{R})}(V_i, W_i)] \nonumber \\
&= \sum_{i=1}^N\frac{\nu(B_i)}{\nu(B)}\cdot \frac{1}{\nu(B_i)}\iint_{B_i \cap (A\times \mathbb{R})}1_{[0, f(x)]}(y)\, \mu(dx)\, dy
\nonumber \\
&= \frac{1}{\nu(B)}\iint_{B \cap (A\times \mathbb{R})}1_{[0, f(x)]}(y)\, \mu(dx)\, dy.
\label{lem_proof_step2} 
\end{align}
Using \eqref{Block_condition_3}, \eqref{lem_proof_step2}, and Fubini's theorem, we obtain \eqref{lem_main_eq}. 
\end{proof}

By Lemma~\ref{lem:main}, 
\eqref{eq_f_densityfunction}, and
\eqref{ineq_nuB_is_greater_than_1}, it follows that
\begin{align}\label{remarkeq_P_is_positive}
P(Z \in \Lambda)=\dfrac{1}{\nu(B)}\displaystyle\int_{E}f(x)\ \mu(dx)=\dfrac{K}{\nu(B)}>0.
\end{align}
Let $P_{Z^{-1}(\Lambda)}$ be the probability measure on $(Z^{-1}(\Lambda), Z^{-1}(\Lambda)\cap\mathcal{F})$ defined by
\begin{align*}
&P_{Z^{-1}(\Lambda)}(A):= \frac{P(A)}{P(Z \in \Lambda)},\qquad 
A\in Z^{-1}(\Lambda)\cap \mathcal{F}:= \{Z^{-1}(\Lambda)\cap F\ |\ F\in \mathcal{F}\}.
\end{align*}
Let $Z|_\Lambda$ denote the restriction of $Z$ to 
$(Z^{-1}(\Lambda), Z^{-1}(\Lambda)\cap \mathcal{F}, P_{Z^{-1}(\Lambda)})$
\begin{align*}
Z : (Z^{-1}(\Lambda), Z^{-1}(\Lambda)\cap \mathcal{F}, P_{Z^{-1}(\Lambda)}) \to (\Lambda, \mathcal{B}(\Lambda)).
\end{align*}
Thus, $Z|_\Lambda$ is a random variable. Let 
$X : (Z^{-1}(\Lambda), Z^{-1}(\Lambda)\cap \mathcal{F}, P_{Z^{-1}(\Lambda)}) \to (E, \mathcal{E})$
be a random variable defined as follows:
\begin{align}
X:= \pi_{\Psi}(Z|_\Lambda) 
=\pi_{\Psi}((U, V_1, V_2, \cdots, V_N, W_1, W_2, \cdots, W_N)|_{\Lambda}).
\label{def:X_patternblock} 
\end{align}
$\displaystyle Z|_{\Lambda}$ and $X$ are conditional random variables defined on 
$(Z^{-1}(\Lambda), Z^{-1}(\Lambda)\cap \mathcal{F}, P_{Z^{-1}(\Lambda)})$. 
Using Lemma~\ref{lem:main} and \eqref{remarkeq_P_is_positive}, we obtain the following theorem:

\begin{thm}\label{thm:main}
For $A\in \mathcal{E}$, we have 
\begin{align*}
P_{Z^{-1}(\Lambda)}(X \in A) = \frac{1}{K}\int_A f(x)\ \mu(dx).
\end{align*}
\end{thm}

Thus, using the conditional random variable $\pi_{\Psi}(Z|_\Lambda)$, 
we can generate random numbers from the density function $\displaystyle \frac{1}{K}f(x)$. 
We call this technique the pattern block method. 
The algorithm is as follows:
\begin{itemize}
\item[Step 1.]
Generate $U$ and $(V_{\Psi(U)}, W_{\Psi(U)})$. 
\item[Step 2.]
If $W_{\Psi(U)} \leq f(V_{\Psi(U)})$, return $V_{\Psi(U)}$.
\item[Step 3.]
If $W_{\Psi(U)} > f(V_{\Psi(U)})$, go to Step~1.
\end{itemize}
Let $R_{{\rm adoption}}$ be the adoption rate defined as
\begin{align}
&R_{{\rm adoption}}:= \frac{K}{\nu(B_1)+\nu(B_2)+\cdots +\nu(B_N)}. 
\label{Def_adoption_rate} 
\end{align}

\section{Applications to generate $1$-dimensional random numbers}\label{Sec_Examples_1dim}

In this section, we assume the following. 
$(E, \mathcal{E})$ is $(\mathbb{R}, \mathcal{B}(\mathbb{R}))$,  
and $\mu$ is the $1$-dimensional Lebesgue measure on $(\mathbb{R}, \mathcal{B}(\mathbb{R}))$. 
Then, $\nu$, defined by \eqref{def_measure_nu}, is 
the $2$-dimensional Lebesgue measure on $(\mathbb{R}^2, \mathcal{B}(\mathbb{R}^2))$. 
For every $C\in \mathcal{B}(\mathbb{R}^2)$, $\vert C\vert$ denotes 
the $2$-dimensional Lebesgue measure $\nu(C)$.

\subsection{The Ziggurat method}\label{Subsec_Ziggurat}

\begin{figure}[h]
\begin{center}
\includegraphics[scale=0.7]{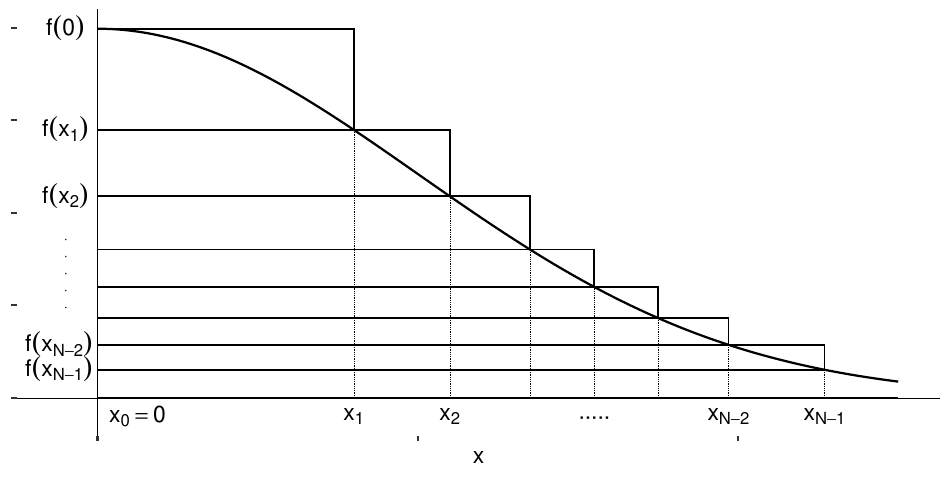}
\caption{Ziggurat method}
\label{fig:Ziggurat}
\end{center}
\end{figure}

In this subsection, we explain that 
the pattern block method is a generalization of the Ziggurat method. 
We assume that the probability density function $f(x)$ strictly decreases for $x > 0$ and satisfies 
$\displaystyle \int_0^{\infty}f(x)\ dx=1$. 
Let $N\in \mathbb{N}$ with $N \geq 2$, 
and let $\displaystyle 0=x_0<x_1<\cdots<x_{N-1}$. Let 
\begin{align*}
&B_i = [0, x_i]\times [f(x_i), f(x_{i-1})]\qquad (i=1, 2, \cdots, N-1), \\
&B_N = ([0, x_{N-1}]\times[0, f(x_{N-1})])
\cup \{(x, y)\in \mathbb{R}^2\ |\ x \geq x_{N-1}, 0\leq y \leq f(x)\}
\end{align*}
be the pattern blocks of $f(x)$. 
The pattern block method then corresponds to the Ziggurat method.

\subsection{Density function with singularities}\label{Subsec_1dim_example1}

Let 
$\displaystyle \varphi(x)=\dfrac{1}{\pi \sqrt{x(1-x)}}$ and 
$\displaystyle h(x)=1+\sin(8\pi x)$. 
Let $f(x)$ be the density function defined as
\begin{align*}
f(x) = 
\left\{ 
\begin{array}{cl}
h(x) \varphi(x)  & (0<x<1), \\
0 & (\text{otherwise}).
\end{array}\right. 
\end{align*}
Let $N=8$. The pattern blocks of $f(x)$ are defined as
\begin{align*}
&B_i = \left\{(x,y)\in\mathbb{R}^2 \;\middle|\;
a_{i-1} \leq x \le a_i,\; 0 \leq y \le b_i \varphi(x)\right\} 
\end{align*}
for $i=1, 2, \cdots , 8$, where
\begin{align*}
&a_i := \dfrac{i}{8}, \quad 
b_i :=
\left\{\begin{array}{cl}
2 & (\text{when $i$ is odd})\\
1 & (\text{when $i$ is even})
\end{array}\right. 
(1\leq i \leq 8).
\end{align*}
Fig.~\ref{Beta_blocks} shows $f(x)$ and $B_1,B_2,\cdots,B_8$.

\begin{figure}[h]
  \centering
  \includegraphics[scale=0.7]{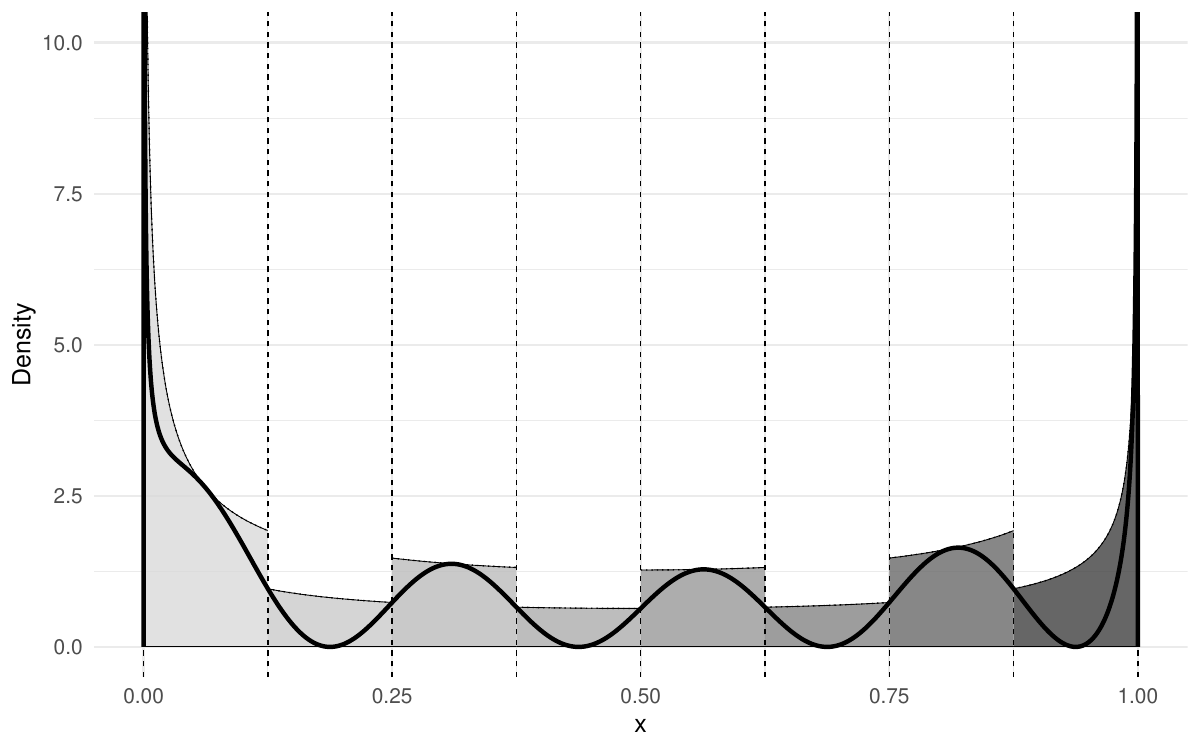}
  \caption{$f(x)$ and its pattern blocks $B_1, B_2, \cdots, B_8$}
  \label{Beta_blocks}
\end{figure}

Because $\varphi(x)$ is the density function of the Beta distribution $\displaystyle B\left(\tfrac{1}{2},\tfrac{1}{2}\right)$,
its cumulative distribution function $\Phi(x)$ is given by 
\begin{align*}
\Phi(x)=\frac{2}{\pi}\arcsin\left(\sqrt{x}\right)\qquad (0<x<1).
\end{align*}
Then $F_{V_i}(x):=P(V_i\leq x)$ is given by 
\begin{align*}
&F_{V_i}(x)
=\left\{
\begin{array}{cl}
0 & (x \le a_{i-1}),\\
\dfrac{b_i}{|B_i|}\bigl(\Phi(x)-\Phi(a_{i-1})\bigr) 
& (a_{i-1} \le x \le a_i),\\
1 & (x \ge a_i) .
\end{array}
\right.
\end{align*}
We assume that the independent random variables $U$, $\xi$, and $\eta$,  
defined on the probability space $(\Omega, \mathcal{F}, P)$, 
are uniformly distributed on $[0,1)$. 
We define the random variables 
$(V_i, W_i)$ $(i=1, 2, \cdots, 8)$ as
\begin{align*}
V_i :=& F_{V_i}^{-1}(\xi)= \Phi^{-1}(\Phi(a_{i-1}) + \xi \, \bigl(\Phi(a_i)-\Phi(a_{i-1})\bigr)) 
=\sin^2\!\left(\frac{\pi}{2}\left\{ \Phi(a_{i-1}) + \xi \, \bigl(\Phi(a_i)-\Phi(a_{i-1})\bigr) \right\} \right), \\
W_i :=& b_i \varphi(V_i)\,\eta = b_i \varphi \big(F_{V_i}^{-1}(\xi)\big)\,\eta .
\end{align*}
Subsequently, $(V_i, W_i)$ $(i=1, 2, \cdots, 8)$ satisfy \eqref{eq_Dist_Vi_Wi}. 
The adoption rate is calculated as follows:
\begin{align*}
&R_{{\rm adoption}}
=\frac{1}{\sum_{i=1}^8 b_i\bigl(\Phi(a_i)-\Phi(a_{i-1})\bigr)}
=\dfrac{2}{3}.
\end{align*}
Thus, by applying the pattern block method, 
we obtained $10000$ random numbers from the probability density function $f(x)$. 
To obtain $10000$ random numbers, $15048$ random number triplets $(U, \xi, \eta)$ were required. 
Fig.~\ref{Beta_result} shows the density scale histogram for $10000$ random numbers.

\begin{figure}[h]
  \centering
  \includegraphics[scale=0.7]{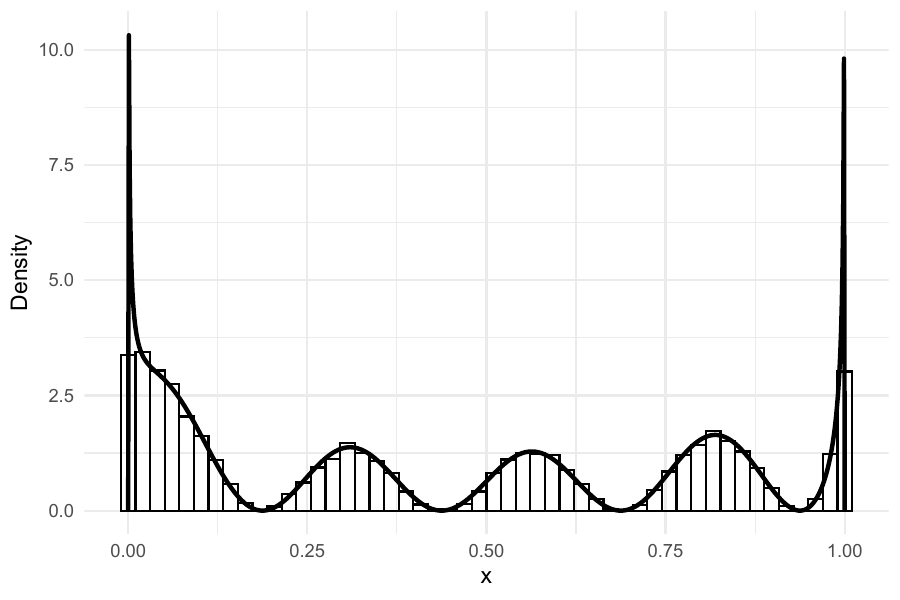}
  \caption{Histogram and density function}
  \label{Beta_result}
\end{figure}

\section{Application to generate $2$-dimensional random numbers}\label{Sec_Examples_2dim}

In this section, we assume that $(E, \mathcal{E})$  is defined as follows:
$E=[\underline{a}, \overline{a}]\times [\underline{a}, \overline{a}]$ and 
$\mathcal{E}=\mathcal{B}(E)$, where $\underline{a}=-4$ and $\overline{a}=4$. 
Let $\mu$ be the $2$-dimensional Lebesgue measure on $(E, \mathcal{E})$. 
Then $\nu$ is the $3$-dimensional Lebesgue measure on 
$(E\times \mathbb{R}, \mathcal{B}(E\times \mathbb{R}))$.
For every $C\in \mathcal{B}(E\times \mathbb{R})$, $\vert C\vert$ denotes 
the $3$-dimensional Lebesgue measure $\nu(C)$.
Let $f(x_1, x_2)$ be the density function on $(E, \mathcal{E})$:
\begin{align*}
&f(x_1, x_2)=
c\Big\{ e^{- x_1^2 - x_2^2}
+\dfrac{1}{2}e^{-(x_1-2)^2-(x_2-2)^2}
\Big\},
\end{align*}
where $c=\dfrac{2119}{9970}$. 
Let
$b_0 = \cfrac{1}{40}$, 
$b_1 = \dfrac{1}{15}$, 
$b_2 = c\big(\dfrac{1}{e^8}+\dfrac{1}{2}\big)$, 
$b_3 = c\big(1+\dfrac{1}{2e^8}\big)$. 
We define $\left[\underline{\kappa}_i, \overline{\kappa}_i\right]$ $(1\leq i \leq 5)$ as
\begin{align*}
&\big[ \underline{\kappa}_1, \overline{\kappa}_1 \big] := \big[0, b_0 \big], \quad
\big[ \underline{\kappa}_2, \overline{\kappa}_2 \big] := \big[ b_0, b_1 \big], \quad
\big[ \underline{\kappa}_3, \overline{\kappa}_3 \big] := \big[ b_1, b_2 \big], \quad
\big[ \underline{\kappa}_4, \overline{\kappa}_4 \big] := \big[ b_1, b_2 \big], \quad
\big[ \underline{\kappa}_5, \overline{\kappa}_5 \big] := \big[ b_2, b_3 \big].
\end{align*}
Moreover, we define $(c_{i,1}, c_{i,2})$ and $d_i$ $(i=3, 4, 5)$ as
\begin{align*}
&(c_{3,1}, c_{3,2})=(c_{5,1}, c_{5,2})=(0,0), \quad (c_{4,1}, c_{4,2})=(2,2), \quad
d_3 =\dfrac{5}{4}, \quad d_4=d_5=1.
\end{align*}
Let $N=5$. The pattern blocks of $f(x_1, x_2)$ are defined as
$B_i = E_i \times \left[\underline{\kappa}_i, \overline{\kappa}_i\right]$ $(1\leq i \leq 5)$, 
where 
\begin{align*}
&E_1=E, \quad E_2 = \left\{(x_1, x_2)\in E\ \middle|\ f(x_1, x_2)\geq b_0 \right\}, \\
&E_i = \left\{(x_1, x_2)\in E\ \middle|\ (x_1-c_{i, 1})^2 + (x_2-c_{i, 2})^2 \leq d_i^2\right\}
\qquad (i=3,4,5).
\end{align*}

\begin{figure}[htbp]
\centering
\begin{minipage}{0.3\linewidth}
    \centering
    \includegraphics[scale=0.5]{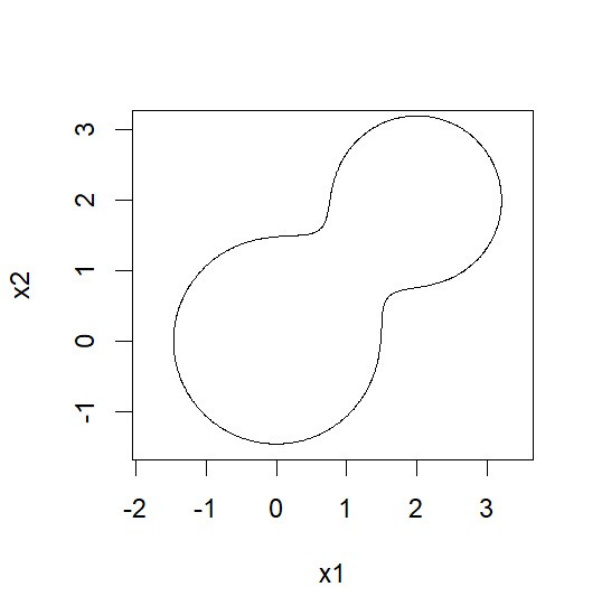}
    \caption{$f(x_1, x_2)=b_0$}
    \label{fig:contour_a0}
\end{minipage}
\hfill
\begin{minipage}{0.3\linewidth}
    \centering
    \includegraphics[scale=0.5]{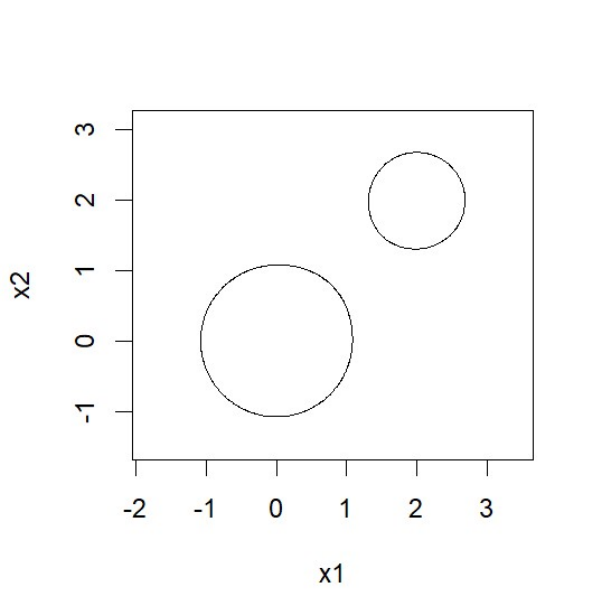}
    \caption{$f(x_1, x_2)=b_1$}
    \label{fig:contour_a1}
\end{minipage}
\hfill
\begin{minipage}{0.3\linewidth}
    \centering
    \includegraphics[scale=0.5]{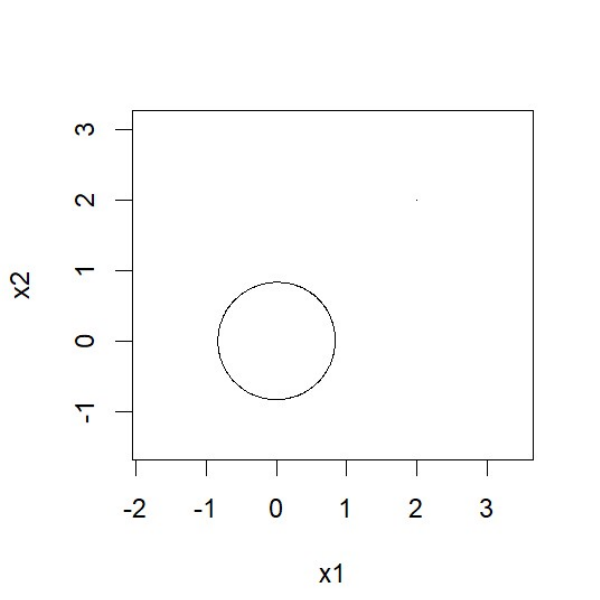}
    \caption{$f(x_1, x_2)=b_2$}
    \label{fig:contour_a2}
\end{minipage}
\end{figure}

We assume that the independent random variables $U$, $\xi_1$, $\xi_2$, and $\eta$,  
defined on the probability space $(\Omega, \mathcal{F}, P)$, are uniformly distributed on $[0,1)$. 
Let 
\begin{align}
&V_1^{(j)} = \underline{a} + (\overline{a} - \underline{a})\cdot \xi_j\quad (j=1,2), \quad 
W_1 = b_0\cdot \eta .
\label{example_2dim_mixednormal_eq1}
\end{align}
Subsequently, $(V_1^{(1)}, V_1^{(2)}, W_1)$ satisfies \eqref{eq_Dist_Vi_Wi}. Let
\begin{align*}
\widetilde{V}_2^{(1)} := \dfrac{11}{2}\xi_1-2,\quad 
\widetilde{V}_2^{(2)} := \dfrac{11}{2}\xi_2-2.
\end{align*}
We define $(V_2^{(1)}, V_2^{(2)}, W_2)$ as 
\begin{align}
&(V_2^{(1)}, V_2^{(2)}) := 
\left(\widetilde{V}_2^{(1)},\widetilde{V}_2^{(2)}\right)\big\vert_{E_2},  \quad
W_2 := (b_1 - b_0)\eta + b_0. \label{example_2dim_mixednormal_eq2}
\end{align}
Subsequently, $(V_2^{(1)}, V_2^{(2)}, W_2)$ satisfies \eqref{eq_Dist_Vi_Wi}. Let 
\begin{align*}
&\Phi_i(r, \theta, y):=(c_{i, 1} + r\cos \theta, c_{i, 2} + r\sin \theta, y), \quad
\widetilde{B}_i := \left[0, d_i\right]\times [0, 2\pi)\times \left[\underline{\kappa}_i, \overline{\kappa}_i\right]
\qquad (i=3, 4, 5).
\end{align*}
Thus, $B_i = \Phi_i(\widetilde{B}_i)$\ $(i=3, 4, 5)$. Let
\begin{align*}
F_{R_i}(r):=
\left\{
\begin{array}{cl}
0 & (r\leq 0) \\
\dfrac{\pi r^2\big(\overline{\kappa}_i-\underline{\kappa}_i\big)}{\vert B_i\vert} & \left(0\leq r\leq d_i\right) \\
1 & \left(r \geq d_i\right).
\end{array}
\right. 
\end{align*}
We define $(V_i^{(1)}, V_i^{(2)}, W_i)$ $(i=3, 4, 5)$ as
\begin{align}
&V_i^{(1)} = c_{i,1}+F_{R_i}^{-1}(\xi_1)\cos 2\pi\xi_2, \quad
V_i^{(2)} = c_{i,2}+F_{R_i}^{-1}(\xi_1)\sin2\pi\xi_2, \quad
W_i = (\overline{\kappa}_i-\underline{\kappa}_i)\eta + \underline{\kappa}_i 
\quad (i=3, 4, 5). \label{example_2dim_mixednormal_eq3}
\end{align}
Subsequently, $(V_i^{(1)}, V_i^{(2)}, W_i)$ $(i=3, 4, 5)$ satisfy \eqref{eq_Dist_Vi_Wi}. 
Thus, by applying the pattern block method, 
we obtained $500000$ random vectors from the probability density function $f(x_1, x_2)$. 
To obtain $500000$ random vectors, $1375334$ random number quartets $(U, V_{\Psi(U)}^{(1)}, V_{\Psi(U)}^{(2)}, W_{\Psi(U)})$ were required. 
The adoption rate is calculated as follows:
\begin{align*}
R_{\rm{adoption}} = \dfrac{1}{\sum_{i=1}^5 \vert B_i\vert} 
=0.3644\cdots .
\end{align*}
Figs.~\ref{fig:2-dim_result} and \ref{fig:density function} show
the density scale histogram for $500000$ random vectors and 
the density function $y=f(x_1, x_2)$.

\begin{figure}[htbp]
\centering
\begin{minipage}{0.49\linewidth}
    \centering
    \includegraphics[scale = 0.5]{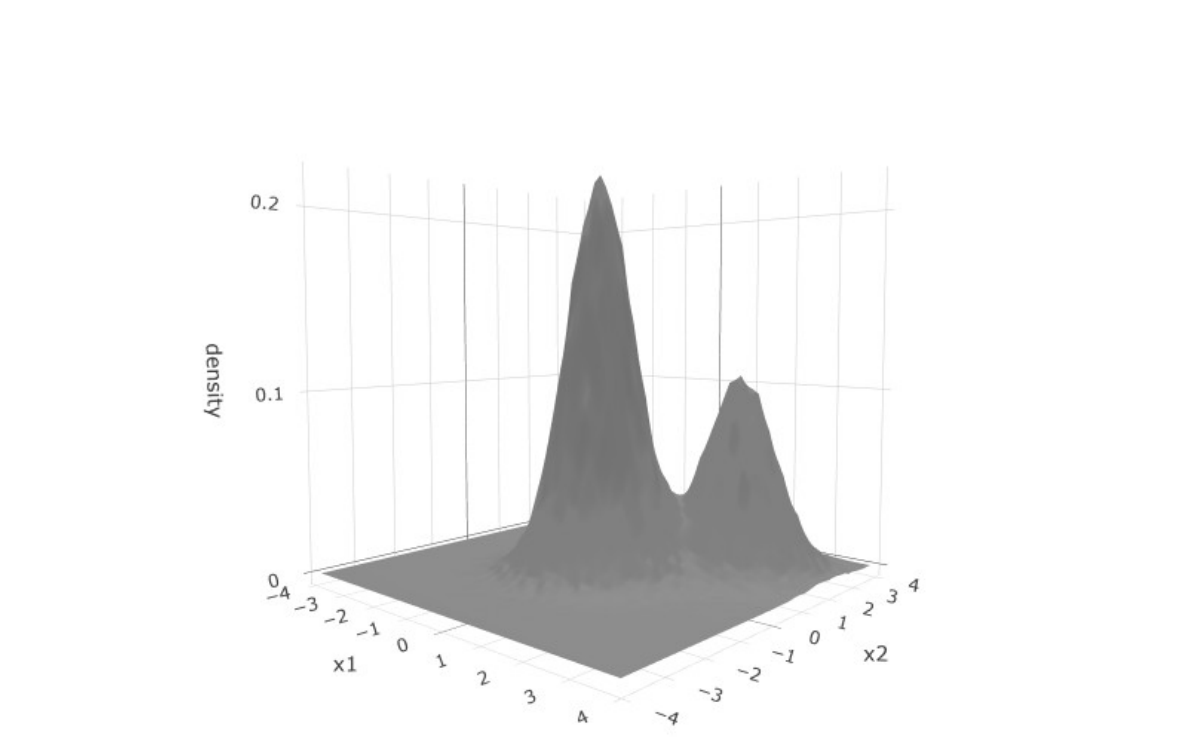}
    \caption{Histogram}
    \label{fig:2-dim_result}
\end{minipage}
\hfill
\begin{minipage}{0.49\linewidth}
    \centering
    \includegraphics[scale = 0.45]{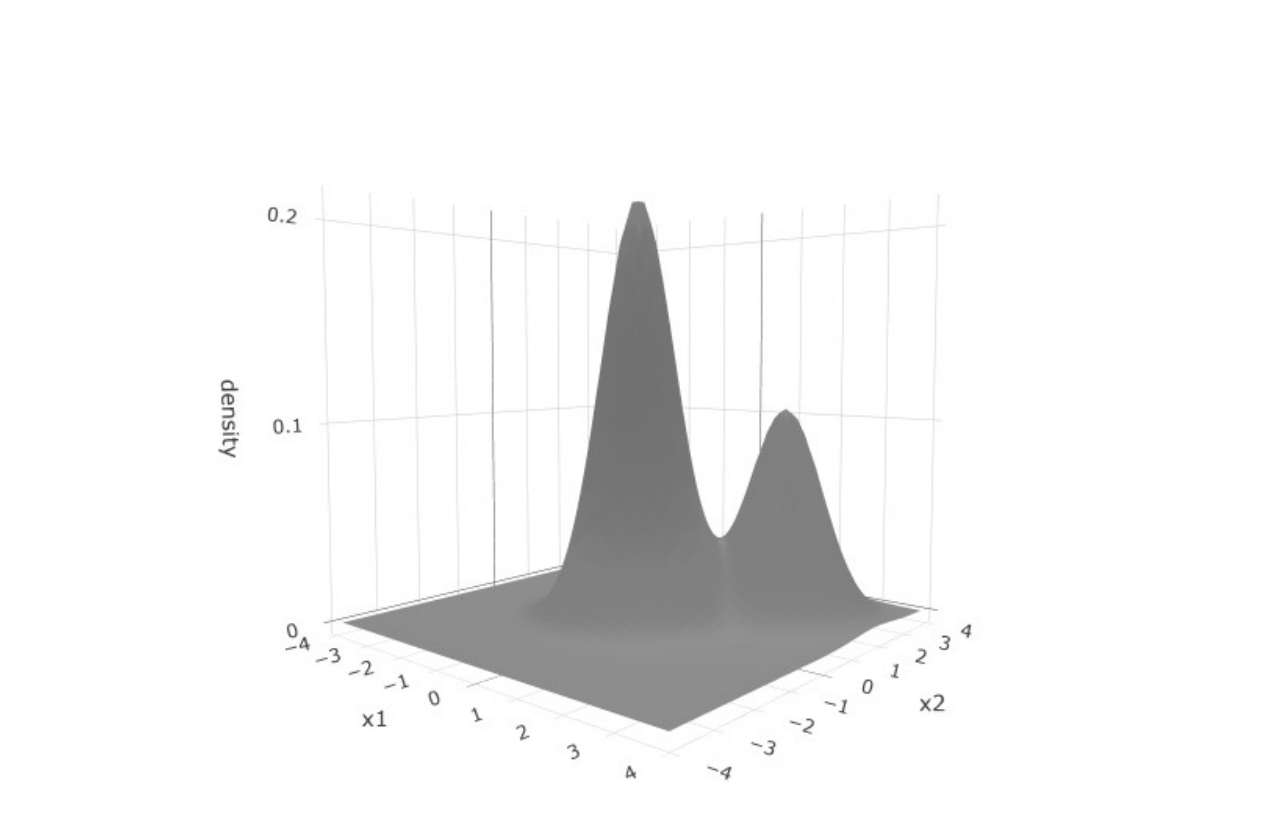} 
    \caption{$y=f(x_1, x_2)$}
    \label{fig:density function}
\end{minipage}
\end{figure}

\section{Conclusion}\label{Sec_conclusion}

This study introduced the pattern block method, a generalization of the Ziggurat method. 
Future work will focus on developing an algorithm 
that identifies the pattern blocks of a given density function with high adoption rate and low search cost.

\section*{Acknowledgments}

The authors thank 
Prof.\ Matsuura Yusuke (National Institute of Advanced Industrial Science and Technology), 
Prof.\ Masaaki Fukasawa (Osaka University) and 
Prof.\ Suguru Yamanaka (Aoyama Gakuin University) 
for their helpful comments and discussions on this study. 
We also thank Editage (www.editage.com) for English language editing. 
This work was supported by JSPS KAKENHI Grant Numbers
JP22K01556 and JP25K05168.

\end{document}